%% file: main.tex
\documentclass[letterpaper, 10 pt, conference]{ieeeconf}  
\IEEEoverridecommandlockouts           

\usepackage[table,usenames,dvipsnames]{xcolor}      
\usepackage[noadjust]{cite}
\usepackage{amsmath,amssymb,amsfonts,amsthm,dsfont,mathtools, nicematrix}

\usepackage{graphicx,tabularx,adjustbox,color}
\usepackage{multirow}
\usepackage[font=small]{caption}
\usepackage[font=footnotesize]{subcaption}
\usepackage{booktabs}
\usepackage{comment}

\usepackage{algorithmic}
\usepackage[ruled,vlined]{algorithm2e}
\usepackage{stackengine}
\def\delequal{\mathrel{\ensurestackMath{\stackon[1pt]{=}{\scriptstyle\Delta}}}}
\allowdisplaybreaks

\usepackage[breaklinks=true, colorlinks, bookmarks=true, citecolor=Black, urlcolor=Violet,linkcolor=Black]{hyperref}

\newtheorem{theorem}{Theorem}[section]
\newtheorem{proposition}[theorem]{Proposition}

\newtheorem{remark}[theorem]{Remark}
\theoremstyle{definition}
\newtheorem{definition}[theorem]{Definition}
\newtheorem{problem}{Problem}

\begin{document}
\input{cls/sym.tex}

\title{\LARGE \bf Safe and Stable Control Synthesis for Uncertain System Models
\\
via Distributionally Robust Optimization} 

\author{Kehan Long$^{*}$ \quad Yinzhuang Yi$^{*}$ \quad Jorge Cort{\'e}s \quad Nikolay Atanasov%
\thanks{$^*$ These authors contributed equally.}%
\thanks{The authors are with the Contextual Robotics Institute, University of California San Diego, La Jolla, CA 92093, USA (e-mails: {\tt \{k3long,yiyi,cortes,natanasov\}@ucsd.edu}).}%
\thanks{We gratefully acknowledge support from NSF RI IIS-2007141.}%
}




%

\maketitle


\begin{abstract}
This paper considers enforcing safety and stability of dynamical systems in the presence of model uncertainty. Safety and stability constraints may be specified using a control barrier function (CBF) and a control Lyapunov function (CLF), respectively. To take model uncertainty into account, robust and chance formulations of the constraints are commonly considered. However, this requires known error bounds or a known distribution for the model uncertainty, and the resulting formulations may suffer from over-conservatism or over-confidence. In this paper, we assume that only a finite set of model parametric uncertainty samples is available and formulate a distributionally robust chance-constrained program (DRCCP) for control synthesis with CBF safety and CLF stability guarantees. To facilitate efficient computation of control inputs during online execution, we present a reformulation of the DRCCP as a second-order cone program (SOCP). Our formulation is evaluated in an adaptive cruise control example in comparison to 1) a baseline CLF-CBF quadratic programming approach, 2) a robust approach that assumes known error bounds of the system uncertainty, and 3) a chance-constrained approach that assumes a known Gaussian Process distribution of the uncertainty. 
\end{abstract}


%
\IEEEpeerreviewmaketitle



%

\input{tex/Intro.tex}
\input{tex/Background}
\input{tex/Problem_formulation.tex}

\input{tex/Approach}
\input{tex/Experiment_formulation.tex}
\input{tex/Conclusion}


\bibliography{ref}
\bibliographystyle{ieeetr}
\appendices
\input{tex/Appendix_qp.tex}
\end{document}

%% file: cls/sym.tex
\newcommand{\calA}{{\cal A}}
\newcommand{\calB}{{\cal B}}
\newcommand{\calC}{{\cal C}}
\newcommand{\calD}{{\cal D}}
\newcommand{\calE}{{\cal E}}
\newcommand{\calF}{{\cal F}}
\newcommand{\calG}{{\cal G}}
\newcommand{\calH}{{\cal H}}
\newcommand{\calI}{{\cal I}}
\newcommand{\calJ}{{\cal J}}
\newcommand{\calK}{{\cal K}}
\newcommand{\calL}{{\cal L}}
\newcommand{\calM}{{\cal M}}
\newcommand{\calN}{{\cal N}}
\newcommand{\calO}{{\cal O}}
\newcommand{\calP}{{\cal P}}
\newcommand{\calQ}{{\cal Q}}
\newcommand{\calR}{{\cal R}}
\newcommand{\calS}{{\cal S}}
\newcommand{\calT}{{\cal T}}
\newcommand{\calU}{{\cal U}}
\newcommand{\calV}{{\cal V}}
\newcommand{\calW}{{\cal W}}
\newcommand{\calX}{{\cal X}}
\newcommand{\calY}{{\cal Y}}
\newcommand{\calZ}{{\cal Z}}

\newcommand{\setA}{\textsf{A}}
\newcommand{\setB}{\textsf{B}}
\newcommand{\setC}{\textsf{C}}
\newcommand{\setD}{\textsf{D}}
\newcommand{\setE}{\textsf{E}}
\newcommand{\setF}{\textsf{F}}
\newcommand{\setG}{\textsf{G}}
\newcommand{\setH}{\textsf{H}}
\newcommand{\setI}{\textsf{I}}
\newcommand{\setJ}{\textsf{J}}
\newcommand{\setK}{\textsf{K}}
\newcommand{\setL}{\textsf{L}}
\newcommand{\setM}{\textsf{M}}
\newcommand{\setN}{\textsf{N}}
\newcommand{\setO}{\textsf{O}}
\newcommand{\setP}{\textsf{P}}
\newcommand{\setQ}{\textsf{Q}}
\newcommand{\setR}{\textsf{R}}
\newcommand{\setS}{\textsf{S}}
\newcommand{\setT}{\textsf{T}}
\newcommand{\setU}{\textsf{U}}
\newcommand{\setV}{\textsf{V}}
\newcommand{\setW}{\textsf{W}}
\newcommand{\setX}{\textsf{X}}
\newcommand{\setY}{\textsf{Y}}
\newcommand{\setZ}{\textsf{Z}}

\newcommand{\bfa}{\mathbf{a}}
\newcommand{\bfb}{\mathbf{b}}
\newcommand{\bfc}{\mathbf{c}}
\newcommand{\bfd}{\mathbf{d}}
\newcommand{\bfe}{\mathbf{e}}
\newcommand{\bff}{\mathbf{f}}
\newcommand{\bfg}{\mathbf{g}}
\newcommand{\bfh}{\mathbf{h}}
\newcommand{\bfi}{\mathbf{i}}
\newcommand{\bfj}{\mathbf{j}}
\newcommand{\bfk}{\mathbf{k}}
\newcommand{\bfl}{\mathbf{l}}
\newcommand{\bfm}{\mathbf{m}}
\newcommand{\bfn}{\mathbf{n}}
\newcommand{\bfo}{\mathbf{o}}
\newcommand{\bfp}{\mathbf{p}}
\newcommand{\bfq}{\mathbf{q}}
\newcommand{\bfr}{\mathbf{r}}
\newcommand{\bfs}{\mathbf{s}}
\newcommand{\bft}{\mathbf{t}}
\newcommand{\bfu}{\mathbf{u}}
\newcommand{\bfv}{\mathbf{v}}
\newcommand{\bfw}{\mathbf{w}}
\newcommand{\bfx}{\mathbf{x}}
\newcommand{\bfy}{\mathbf{y}}
\newcommand{\bfz}{\mathbf{z}}

\newcommand{\bfalpha}{\boldsymbol{\alpha}}
\newcommand{\bfbeta}{\boldsymbol{\beta}}
\newcommand{\bfgamma}{\boldsymbol{\gamma}}
\newcommand{\bfdelta}{\boldsymbol{\delta}}
\newcommand{\bfepsilon}{\boldsymbol{\epsilon}}
\newcommand{\bfzeta}{\boldsymbol{\zeta}}
\newcommand{\bfeta}{\boldsymbol{\eta}}
\newcommand{\bftheta}{\boldsymbol{\theta}}
\newcommand{\bfiota}{\boldsymbol{\iota}}
\newcommand{\bfkappa}{\boldsymbol{\kappa}}
\newcommand{\bflambda}{\boldsymbol{\lambda}}
\newcommand{\bfmu}{\boldsymbol{\mu}}
\newcommand{\bfnu}{\boldsymbol{\nu}}
\newcommand{\bfomicron}{\boldsymbol{\omicron}}
\newcommand{\bfpi}{\boldsymbol{\pi}}
\newcommand{\bfrho}{\boldsymbol{\rho}}
\newcommand{\bfsigma}{\boldsymbol{\sigma}}
\newcommand{\bftau}{\boldsymbol{\tau}}
\newcommand{\bfupsilon}{\boldsymbol{\upsilon}}
\newcommand{\bfphi}{\boldsymbol{\phi}}
\newcommand{\bfchi}{\boldsymbol{\chi}}
\newcommand{\bfpsi}{\boldsymbol{\psi}}
\newcommand{\bfomega}{\boldsymbol{\omega}}
\newcommand{\bfxi}{\boldsymbol{\xi}}
\newcommand{\bfell}{\boldsymbol{\ell}}

\newcommand{\bfA}{\mathbf{A}}
\newcommand{\bfB}{\mathbf{B}}
\newcommand{\bfC}{\mathbf{C}}
\newcommand{\bfD}{\mathbf{D}}
\newcommand{\bfE}{\mathbf{E}}
\newcommand{\bfF}{\mathbf{F}}
\newcommand{\bfG}{\mathbf{G}}
\newcommand{\bfH}{\mathbf{H}}
\newcommand{\bfI}{\mathbf{I}}
\newcommand{\bfJ}{\mathbf{J}}
\newcommand{\bfK}{\mathbf{K}}
\newcommand{\bfL}{\mathbf{L}}
\newcommand{\bfM}{\mathbf{M}}
\newcommand{\bfN}{\mathbf{N}}
\newcommand{\bfO}{\mathbf{O}}
\newcommand{\bfP}{\mathbf{P}}
\newcommand{\bfQ}{\mathbf{Q}}
\newcommand{\bfR}{\mathbf{R}}
\newcommand{\bfS}{\mathbf{S}}
\newcommand{\bfT}{\mathbf{T}}
\newcommand{\bfU}{\mathbf{U}}
\newcommand{\bfV}{\mathbf{V}}
\newcommand{\bfW}{\mathbf{W}}
\newcommand{\bfX}{\mathbf{X}}
\newcommand{\bfY}{\mathbf{Y}}
\newcommand{\bfZ}{\mathbf{Z}}

\newcommand{\bfGamma}{\boldsymbol{\Gamma}}
\newcommand{\bfDelta}{\boldsymbol{\Delta}}
\newcommand{\bfTheta}{\boldsymbol{\Theta}}
\newcommand{\bfLambda}{\boldsymbol{\Lambda}}
\newcommand{\bfPi}{\boldsymbol{\Pi}}
\newcommand{\bfSigma}{\boldsymbol{\Sigma}}
\newcommand{\bfUpsilon}{\boldsymbol{\Upsilon}}
\newcommand{\bfPhi}{\boldsymbol{\Phi}}
\newcommand{\bfPsi}{\boldsymbol{\Psi}}
\newcommand{\bfOmega}{\boldsymbol{\Omega}}

\newcommand{\bbA}{\mathbb{A}}
\newcommand{\bbB}{\mathbb{B}}
\newcommand{\bbC}{\mathbb{C}}
\newcommand{\bbD}{\mathbb{D}}
\newcommand{\bbE}{\mathbb{E}}
\newcommand{\bbF}{\mathbb{F}}
\newcommand{\bbG}{\mathbb{G}}
\newcommand{\bbH}{\mathbb{H}}
\newcommand{\bbI}{\mathbb{I}}
\newcommand{\bbJ}{\mathbb{J}}
\newcommand{\bbK}{\mathbb{K}}
\newcommand{\bbL}{\mathbb{L}}
\newcommand{\bbM}{\mathbb{M}}
\newcommand{\bbN}{\mathbb{N}}
\newcommand{\bbO}{\mathbb{O}}
\newcommand{\bbP}{\mathbb{P}}
\newcommand{\bbQ}{\mathbb{Q}}
\newcommand{\bbR}{\mathbb{R}}
\newcommand{\bbS}{\mathbb{S}}
\newcommand{\bbT}{\mathbb{T}}
\newcommand{\bbU}{\mathbb{U}}
\newcommand{\bbV}{\mathbb{V}}
\newcommand{\bbW}{\mathbb{W}}
\newcommand{\bbX}{\mathbb{X}}
\newcommand{\bbY}{\mathbb{Y}}
\newcommand{\bbZ}{\mathbb{Z}}

\newcommand{\ubfu}{\underline{\bfu}}
\newcommand{\Var}{\textit{Var}}
\newcommand{\Cov}{\textit{Cov}}

%% file: tex/Intro.tex
\section{Introduction}
\label{sec: intro}

With the increasing deployment of automatic control systems and robotic platforms in unstructured real-world environments, it is crucial to develop feedback controllers with safety and stability guarantees in the presence of model uncertainty. Enforcing safety by utilizing set invariance properties has become a mainstream approach for constrained control synthesis. Inspired by the property of control Lyapunov functions (CLFs) \cite{sontag1989universal} to yield invariant level sets, control barrier functions (CBFs) \cite{wieland2007} were introduced as a tool to verify that a desired safe subset of the state space is invariant. Stability and safety can be considered simultaneously by introducing CLF and CBF constraints on the control input in a quadratic program (QP) formulation for control synthesis \cite{ames2014cdc, ames2016control}. The reliability and efficiency of CLF-CBF-QP control synthesis has been evidenced in several robotic applications, including multi-agent systems \cite{Wang2017TRO}, aerial robots \cite{Wu2016ACC}, and walking robots \cite{nguyen2016cdc}.

The notion of safety in the presence of system model uncertainty has been mainly described in two ways: using robust constraints \cite{freeman_robust, Petersen2014RobustCO} or chance constraints \cite{Masahiro_ccdp, Zhu2019ChanceConstrainedCA}.  Studies have also considered system uncertainty when pairing safety with stability in the CLF-CBF-QP formulation. Regarding robust formulations, Choi et al. \cite{Jason2020cdc} consider model disturbances with a compact and convex support set and propose a robust control barrier value function to ensure safety. Similarly, \cite{Nguyen2022RobustSC} assumes bounded model uncertainty and reformulates the original safety and stability constraints as min-max constraints. Regarding probabilistic formulations, \cite{dhiman2020control,Long2022RAL} assume a Gaussian Process distribution for the model uncertainty and propose probabilistic versions of the CLF stability and CBF safety constraints. All these approaches require known error bounds or known distributions of the uncertainty. In addition, robust formulations may suffer from over-conservatism due to the worst-case error bounds, while chance-constrained formulations may suffer from over-confidence due to a distributional shift at deployment time. 

To tackle such scenarios, we rely on a body of work from the literature on stochastic programming \cite{AS-DD-AR:14} that considers distributionally robust versions of stochastic optimization problems, see e.g. \cite{AB-LEG-AN:09,AS:17}. In particular, distributionally robust chance-constrained programs (DRCCP) deal with uncertain variables in the constraints when only finitely many samples are available. The main idea is to construct an ambiguity ball centered at the empirical distribution obtained from the observed samples and with radius defined using a probability distance function, such as Kullback–Leibler divergence \cite{Jiang2016DatadrivenCC} or Wasserstein distance \cite{Esfahani2018DatadrivenDR,Chen2018DataDrivenCC,Xie2021OnDR,Hota2019DataDrivenCC,DB-JC-SM:21-tac}. In DRCCP, the desired constraints must be satisfied with high probability for all distributions in the constructed ambiguity set. Given the ability to handle uncertainty with unknown or shifting distribution within the ambiguity set, distributionally robust formulations have been used to enforce constraints in reinforcement learning \cite{chow2017risk, smirnova2019distributionally} and Markov decision processes \cite{DR_MDP_2012, chen2019distributionally, petrik2019beyond}. While these works are closely related, their focus is on discrete-time planning with robustness to uncertainty, while our work considers continuous-time control with safety and stability guarantees.

The contributions of this work are summarized as follows. First, we relax the assumption for safe and stable control synthesis that known error bounds or known distribution of model uncertainty are available by formulating distributionally robust safety and stability constraints using offline model uncertainty samples. Second, we show that the DRCCP control synthesis problem can be reformulated as a second-order cone program (SOCP) in two cases: when there is no restriction on the uncertainty support set and when the uncertainty support set is polyhedral. We demonstrate on an adaptive cruise control problem how our DRCCP SOCP guarantees safety in scenarios with incorrect model uncertainty error bounds or uncertainty distribution shift, in contrast with the vanilla CLF-CBF-QP approach, a robust approach, and a chance-constrained approach.

%% file: tex/Background.tex
\section{Preliminaries}
\label{sec: prelim}



This section reviews control Lyapunov and control barrier functions, distributionally robust modeling, and chance-constrained programming. 

\subsection{Optimization-based Control Synthesis}
Consider a non-linear control-affine system\footnotemark[1]{}:
\begin{equation}
\label{eq: true_dynamic}
\begin{aligned}
    \dot{\bfx} = f(\bfx) + g(\bfx) \bfu =     [f(\bfx) \; g(\bfx)] \cdot\begin{bmatrix}
    1 \\
    \bfu
    \end{bmatrix} \delequal F(\bfx)\ubfu,
\end{aligned}
\end{equation}
where $\boldsymbol{x} \in \calX \subseteq \mathbb{R}^{n}$ is the state and $\ubfu \in \underline{\mathcal{U}} := \{ 1\} \times \mathbb{R}^{m}$ is the control input. Assume $f : \mathbb{R}^{n} \mapsto \mathbb{R}^{n}$ and $g : \mathbb{R}^{n} \mapsto \mathbb{R}^{n \times m}$ are locally Lipschitz.
We start by recalling the notions of CLF \cite{sontag1989universal} and CBF \cite{ames2016control}, which play a key role in the synthesis of stable and safe controllers, respectively.

\footnotetext[1]{\textbf{Notation. } The sets of real, non-negative real, and natural numbers are denoted by $\bbR$, $\bbR_{\geq 0}$, and $\bbN$, respectively. For $N \in \bbN$, we let $[N] := \{1,2, \dots N\}$. We denote the distribution and expectation of a random variable $Y$ by $\mathbb{P}$ and $\bbE_{\bbP}(Y)$, respectively. We use $\boldsymbol{0}_n$ and $\boldsymbol{1}_n$ to denote the $n$-dimensional vector with all entries equal to $0$ and $1$, respectively. For scalar $x$, we define $(x)_+ := \max(x,0)$. The $L_2$ norm for a vector $\bfx$ is denoted by $\|\bfx\|$. We denote by $\bfI_m \in \bbR^{m \times m}$ the identity matrix and by $\otimes$ the Kronecker product. We use $\text{vec}(\bfX) \in \mathbb{R}^{nm}$ to denote the vectorization of $\bfX \in \mathbb{R}^{n \times m}$, obtained by stacking its columns. The gradient of a differentiable function $V$ is denoted by $\nabla V$, while its Lie derivative along a vector field $f$ by $\calL_f V  = \nabla V \cdot f$. A continuous function $\alpha: [0,a)\rightarrow [0,\infty )$ is of class $\calK$ if it is strictly increasing and $\alpha(0) = 0$. A continuous function $\alpha:\mathbb{R} \rightarrow \mathbb{R}$ is of extended class $\calK_{\infty}$ if it is of class $\calK$ and $\lim_{r \rightarrow \infty} \alpha(r) = \infty$.}



\begin{definition}
A positive-definite continuously differentiable function $V: \mathbb{R}^n \mapsto {\mathbb{R}_{\geq 0}}$ is a \emph{control Lyapunov function (CLF)} on $\calX$ for system \eqref{eq: true_dynamic} if there exists a class $\mathcal{K}$ function $\alpha_V$ such that:
\begin{equation}\label{eq:clf}
    \inf_{\ubfu \in \underline{\mathcal{U}}} \textit{CLC}(\bfx,\ubfu) \leq 0, \quad \forall \bfx \in \calX \setminus \{\mathbf{0}\},
\end{equation}
where the \emph{control Lyapunov condition (CLC)} is:
\begin{equation}\label{eq:clc_define}
\begin{aligned}
    \textit{CLC}(\bfx,\ubfu) := \mathcal{L}_f V(\bfx) + \mathcal{L}_g V(\bfx)\bfu + \alpha_V( V(\bfx)).
\end{aligned}
\end{equation}
\end{definition}

The existence of a CLF simplifies the stabilization problem considerably because a stabilizing feedback control law can be obtained in terms of the derivatives of the CLF \cite{sontag1989universal}.

In addition to stability, it is often necessary to ensure that the closed-loop system trajectories remain within a safe set $\calC \subset \calX$. To facilitate safe control synthesis, the safe set is specified as the zero superlevel set, $\calC := \{\bfx \in \bbR^n : h(\bfx) \geq 0 \}$, of a function $h$. 


\begin{definition}
A continuously differentiable function $h: \mathbb{R}^n \mapsto {\mathbb{R}}$ is a \emph{control barrier function (CBF)} on $\mathcal{X}$ for system \eqref{eq: true_dynamic} if there exists an extended class $\mathcal{K}_{\infty}$ function $\alpha_h$ with:
\begin{equation}\label{eq:cbf}
    \sup_{\ubfu\in \underline{\mathcal{U}}} \textit{CBC}(\bfx,\ubfu) \geq 0, \quad \forall \bfx \in \calX,
\end{equation}
where the \emph{control barrier condition (CBC)} is:
\begin{equation}
\label{eq:cbc_define}
\begin{aligned}
    \textit{CBC}(\bfx,\ubfu)  := \mathcal{L}_f h(\bfx) + \mathcal{L}_g h(\bfx)\bfu + \alpha_h (h(\bfx)). 
\end{aligned}
\end{equation}
\end{definition}


Noting that the CLF stability requirement in \eqref{eq:clf} and the CBF safety requirement in \eqref{eq:cbf} are affine in $\ubfu$, they can be enforced as constraints in an optimization problem. Given a baseline controller $\underline{\bfk}(\bfx)$, the following QP modifies the controller to guarantee safety and encourage stability:
\begin{equation}
\label{eq: QP_origin}
\begin{aligned}
& \min_{\ubfu \in \underline{\calU},\delta \in \bbR_{\geq 0}}\,\, \|\ubfu - \underline{\bfk}(\bfx))\|^2 + \lambda \delta^2\\
\mathrm{s.t.} \, \,  &\textit{CLC}(\bfx,\ubfu) \leq \delta,  \textit{CBC}(\bfx,\ubfu) \geq 0,
\end{aligned}
\end{equation}
%
where $\delta \in \bbR_{\geq 0}$ is a slack variable that relaxes the CLF constraints to ensure the feasibility of the QP, controlled by the scaling factor $\lambda > 0$.

We are interested in the control synthesis problem in \eqref{eq: QP_origin} when the system dynamics in \eqref{eq: true_dynamic} are not perfectly known. Considering probabilistic uncertainty in the system model requires probabilistic versions of the safety and stability constraints in \eqref{eq: QP_origin}. We investigate how to handle model uncertainty using samples rather than a known distribution and whether the uncertainty-aware versions of the constraints in \eqref{eq: QP_origin} remain convex and tractable.

\subsection{Distributionally Robust Chance-constrained Program}\label{sec: drccp_prelim}

To handle probabilistic constraints, we begin by reviewing chance-constrained programming. Throughout the paper we consider a complete separable metric space $\Xi$ with metric $d$ and associate with it a Borel $\sigma$-algebra $\calF$ and the set $\calP(\Xi)$ of Borel probability measures on $\Xi$.  A chance-constrained program (CCP) takes the form:
\begin{equation}
\begin{aligned}
\label{eq: ccp}
    &\min_{\bfz \in \calZ} \bfc^\top \bfz,  \\
    \mathrm{s.t.} \, \, & \mathbb{P}(G(\bfz, \bfxi) \leq 0) \geq 1 - \epsilon, 
\end{aligned}
\end{equation}
with closed convex set $\calZ \subseteq \bbR^n$ and uncertainty set $\Xi \subseteq \bbR^k$. The constraint function $G(\bfz, \bfxi) \in \calZ \times \Xi \mapsto \mathbb{R}$ depends both on the decision vector $\bfz$ and an uncertainty vector $\bfxi$, whose distribution $\bbP$ is supported on $\Xi$, and $\epsilon \in (0,1)$ is a user-specified risk tolerance. 
The feasible set defined by the chance constraint in \eqref{eq: ccp} is not convex in general. Nemirovski and Shapiro \cite{Nemirovski2006ConvexAO} proposed a conservative convex approximation \cite{Nemirovski2006ConvexAO} of the feasible set in \eqref{eq: ccp}, which consists of replacing the chance constraint by a conditional value-at-risk (CVaR) constraint:
\begin{equation}
\label{eq: cvar_ccp}
\begin{aligned}
    &\min_{\bfz \in \calZ} \bfc^\top \bfz,  \\
    \mathrm{s.t.} \, \, & \textit{CVaR}_{1-\epsilon}^{\mathbb{P}}(G(\bfz,\bfxi)) \leq 0.
\end{aligned}
\end{equation}
The feasible set of \eqref{eq: cvar_ccp} is a subset of the feasible set of~\eqref{eq: ccp}. The following paragraph describes a way of defining CVaR.

Value-at-risk (VaR) at confidence level $1 - \epsilon$ for $\epsilon \in (0,1)$
is defined as $\textit{VaR}_{1-\epsilon}^{\mathbb{P}_q}(Q) := \inf_{t \in \bbR}\{t \; | \; \bbP_q(Q \leq t) \geq 1 - \epsilon\}$ for a random variable $Q$ with distribution $\bbP_q$.  VaR does not provide information about the right tail of the distribution, and optimization programs involving VaR variables are intractable in general \cite{Mausser_1999}. To address this, Rockafellar and Uryasev \cite{Rockafellar00optimizationof}
introduced conditional value-at-risk (CVaR), defined as $\textit{CVaR}_{1-\epsilon}^{\mathbb{P}_q}(Q) = \bbE_{\mathbb{P}_q} [ Q \; | \; Q \geq \textit{VaR}_{1-\epsilon}^{\mathbb{P}_q}(Q)] $. CVaR can be also formulated as a convex program:
\begin{equation}
\label{eq: cvar_opti_def}    
    \textit{CVaR}_{1-\epsilon}^{\mathbb{P}_q}(Q) := \inf_{t \in \mathbb{R}}[\epsilon^{-1}\mathbb{E}_{\mathbb{P}_q}[(Q+t)_+]-t].
\end{equation}

Both the formulations in \eqref{eq: ccp} and \eqref{eq: cvar_ccp} assume that $\bbP$, the true distribution of $\bfxi$, is known. When this is not the case, one can instead resort to distributionally robust formulations~\cite{Esfahani2018DatadrivenDR,Xie2021OnDR}. 
Assume we only have access to samples $\{\bfxi_i\}_{i \in [N]}$  from the true distribution of $\bfxi$. We describe a way of constructing an ambiguity set of distributions that could have potentially generated such samples. Let $\calP_p(\Xi) \subseteq \calP(\Xi)$ be the set of Borel probability measures with finite $p$-th moment for $p \geq 1$. The $p$-Wasserstein distance \cite{ Hota2019DataDrivenCC} between two probability measures $\mu$, $\nu$ in $\calP_p(\Xi)$ is:
\begin{equation}
\label{eq: wasserstein_def}
    W_{p}(\mu,\nu) := \left(\inf_{\gamma \in \bbQ(\mu,\nu)} \left[ \int_{\Xi \times \Xi} d(x,y)^p \text{d}\gamma(x,y) 
    \right] \right)^{\frac{1}{p}}, 
\end{equation}
where $\bbQ(\mu,\nu)$ denotes the measures on $\Xi \times \Xi$ with marginals $\mu$ and $\nu$ on the first and second factors, and $d$ denotes the metric in the space $\Xi$.

%
%

Let $\hat{\mathbb{P}}_N :=  \frac{1}{N}\sum_{i=1}^N \delta_{\bfxi_i}$ denote the discrete empirical distribution constructed from the observed samples $\{\bfxi_i\}_{i=1}^N$. Using the Wasserstein distance \eqref{eq: wasserstein_def}, one can define a Wasserstein ambiguity set of radius $r$ centered at $\hat{\mathbb{P}}_N$: 
\begin{equation} \label{eq:ambiguity-set}
      \calM_{N}^{r} := \{\mu \in \calP_p(\Xi) \; | \; W_p(\mu,\hat{\mathbb{P}}_{N} ) \leq r\},
\end{equation}
and, in turn, a distributionally robust chance-constrained program (DRCCP):
\begin{equation}
\begin{aligned}
\label{eq: drccp_def}
    &\min_{\bfz \in \calZ} \bfc^\top \bfz,  \\
    \mathrm{s.t.} \, \, & \inf_{\bbP \in \calM_{N}^{r}}\mathbb{P}(G(\bfz, \bfxi) \leq 0) \geq 1 - \epsilon. 
\end{aligned}
\end{equation}
The constraint in \eqref{eq: drccp_def} is equivalent to 
$\sup_{\bbP \in \calM_{N}^{r}}\mathbb{P}(G(\bfz, \bfxi) \geq 0) \leq \epsilon$.
Thus, mimicking the convexification for CCP in \eqref{eq: cvar_ccp}, one can use CVaR to obtain a convex approximation of \eqref{eq: drccp_def}: 
\begin{equation}
\begin{aligned}
\label{eq: drccp_cvar_notation}
    &\min_{\bfz \in \calZ} \bfc^\top \bfz,  \\
    \mathrm{s.t.} \, \, & \sup_{\bbP \in \calM_{N}^{r}} \textit{CVaR}_{1-\epsilon}^{\mathbb{P}}(G(\bfz,\bfxi)) \leq 0.
\end{aligned}
\end{equation}

%% file: tex/Problem_formulation.tex
\section{Problem Formulation}\label{sec: problem}

%

We study the problem of enforcing safety and stability of control-affine dynamical systems with model uncertainty. Critically, we do not assume that the probability distribution or error bounds for the model uncertainty are known. We model the uncertainty in the system in \eqref{eq: true_dynamic} using a nominal model $\tilde{F}(\bfx)$ and a linear combination of $k$ perturbations:
\begin{equation}\label{eq: uncertain_dynamic_1}
    \dot{\bfx} = F(\bfx)\ubfu = (\tilde{F}(\bfx) + \sum_{j=1}^k W_j(\bfx) \xi_j)\ubfu.
\end{equation}
For $1 \leq j \leq k$, we use $W_j(\bfx) \in \mathbb{R}^{n \times (m+1)}$ to denote the possible model perturbations, and $\bfxi \in \bbR^k$ with elements $\xi_j \in \mathbb{R}$ to denote the corresponding unknown weights. We assume the perturbations $W_j(\bfx)$ are known, while the weights $\bfxi$ are stochastic. We require a set of historical realizations $\{\bfxi_i\}_{i=1}^N$ as training data, which can be obtained from past state-control system trajectories using system identification techniques, e.g., based on neural ordinary differential equations \cite{thai_2022_learn_disturbance} or Koopman operator theory \cite{klus2020data}.


Many control applications require safety and stability guarantees for an uncertain system under online error realizations.
This motivates us to consider a distributionally robust formulation for online control synthesis. 



\begin{problem}[\textbf{Distributionally Robust Safety and Stability for Uncertain Systems}]\label{prob: data_driven_safety_stability}
Consider a nominal model $\tilde{F}(\bfx)$ and perturbation matrices $W_j(\bfx)$, $j \in [k]$ for the system dynamics in \eqref{eq: true_dynamic}. Given observations $\{\bfxi_i\}_{i=1}^N$ of the model uncertainty $\bfxi$ with support set $\Xi$, design a feedback controller $\underline{\bfk}^*:\bbR^n \mapsto \underline{\calU}$ with a risk-tolerance parameter $\epsilon \in (0,1)$ such that, for each $\bfx \in \calX$:
\begin{align}
    &\inf_{\mathbb{P} \in \calM_{N}^{r_1}}\mathbb{P}(\textit{CLC}(\bfx,\underline{\bfk}^*(\bfx), \bfxi)) \leq \delta ) \geq 1 - \epsilon, \notag\\
    &\inf_{\mathbb{P} \in \calM_{N}^{r_2}}\mathbb{P}(\textit{CBC}(\bfx,\underline{\bfk}^*(\bfx), \bfxi)) \geq 0 ) \geq 1 - \epsilon,
\end{align}
where $\calM_{N}^{r_1}$, $\calM_{N}^{r_2}$ are Wasserstein ambiguity sets with user-specified radii $r_1$ and $r_2$.
\end{problem}

While we do not assume a particular distribution for $\bfxi$, the Wasserstein ball radii $r_1$, $r_2$ specify the maximal shift of the true distribution of $\bfxi$ from the empirical distribution $\hat{\mathbb{P}}_N$ of the historical samples that our method can handle.

We consider two cases based on the information available about the support set $\Xi$. In the first case, we consider a  unbounded support set $\Xi = \bbR^k$; in the second case, we assume a compact polyhedron set $\Xi = \{\bfxi \in \mathbb{R}^k \;|\; \bfC \bfxi \leq \bfd\}$. Inspired by the CLF-CBF-QP in \eqref{eq: QP_origin}, we consider the following DRCCP formulation to enforce safety and stability with high probability and out-of-sample errors by leveraging the CVaR approximations \eqref{eq: drccp_cvar_notation} and the CVaR  definition~\eqref{eq: cvar_opti_def},
%
%
%
\begin{align}
\label{eq: clc_cbc_cvar_approximate}
& \min_{\ubfu \in \underline{\calU},\delta \in \bbR}\,\, \|\ubfu - \underline{\bfk}(\bfx)\|^2 + \lambda \delta^2 \\
\mathrm{s.t.} \, \, &\sup_{\mathbb{P}\in \calM_{N}^{r_1}}\inf_{t \in \mathbb{R}}[\epsilon^{-1}\mathbb{E}_{\bbP}[(\textit{CLC}(\bfx,\ubfu,\bfxi)+t-\delta)_{+}] -t] \leq 0, \notag \\
&\sup_{\mathbb{P}\in \calM_{N}^{r_2}}\inf_{t \in \mathbb{R}}[\epsilon^{-1}\mathbb{E}_{\bbP}[(-\textit{CBC}(\bfx,\ubfu,\bfxi)+t)_{+}] -t]  \!\leq \! 0  \notag .
\end{align}
Although the constraints in~\eqref{eq: clc_cbc_cvar_approximate} are convex, the program is intractable~\cite{Hota2019DataDrivenCC, Esfahani2018DatadrivenDR} due to the search of suprema over the Wasserstein ambiguity set. In the following sections, we discuss our approach to identify tractable reformulations of~\eqref{eq: clc_cbc_cvar_approximate} and enable online stable and safe control synthesis.

%% file: tex/Approach.tex
\section{Tractable Reformulation of Control Synthesis With Model Uncertainty}
\label{sec: method}


This section presents our approach for solving \eqref{eq: clc_cbc_cvar_approximate}. To simplify the notation, we use the vectorization of $F(\bfx)$,
\begin{align}
\label{eq: F_vectorization}
    \text{vec}(F(\bfx))
    = \text{vec}(\tilde{F}(\bfx)) + \bfW(\bfx) \bfxi,
\end{align}
where 
\begin{align*}
    \bfW(\bfx) &= \left[
    \text{vec}(W_1(\bfx)) \; \; \cdots \;\;  \text{vec}(W_k(\bfx)) \right] \in \bbR^{n(m+1) \times k}.
\end{align*}
Observe that the $\textit{CBC}$ expression in \eqref{eq:cbc_define} is affine in both $\ubfu$ and $\bfxi$. 
Using the Kronecker product property $\text{vec}(\bfA\bfB\bfC) = (\bfC^\top \otimes \bfA)\text{vec}(\bfB)$ and $\text{vec}(F(\bfx))$ in \eqref{eq: F_vectorization}, we have:
\begin{align}
\label{eq: cbc_kronecker}
    &\textit{CBC}(\bfx,\ubfu,\bfxi) = [\nabla_{\bfx} h(\bfx)]^\top F(\bfx)\ubfu + \alpha_h(h(\bfx)) \notag \\
     &= \ubfu^\top \text{vec}([\nabla_{\bfx} h(\bfx)]^\top F(\bfx)\bfI_{m+1}) + \alpha_h(h(\bfx)) \notag \\
    &= \ubfu^\top(\bfI_{m+1} \otimes [\nabla_{\bfx} h(\bfx)]^\top)\text{vec}(F(\bfx)) + \alpha_h(h(\bfx)) \notag \\
    & = \ubfu^\top\underbrace{(\bfI_{m+1} \otimes [\nabla_{\bfx} h(\bfx)]^\top)\text{vec}(\tilde{F}(\bfx))}_{\bfq_h(\bfx)}+ \alpha_h(h(\bfx)) + \notag \\
    & \qquad \ubfu^\top \underbrace{(\bfI_{m+1} \otimes [\nabla_{\bfx} h(\bfx)]^\top) \bfW(\bfx)}_{\bfR_h(\bfx)} \bfxi  \notag \\
    & = \ubfu^\top \bfq_h(\bfx) + \ubfu^\top \bfR_h(\bfx) \bfxi + \alpha_h(h(\bfx)).
\end{align}
We can also write $\textit{CLC}(\bfx,\ubfu,\bfxi) = \ubfu^\top \bfq_V(\bfx) +  \ubfu^\top \bfR_V(\bfx) \bfxi + \alpha_V(V(\bfx))$ with similar definitions. Since $\tilde{F}(\bfx)$, $h(\bfx)$, $V(\bfx)$, and $\bfW(\bfx)$ are known and deterministic, both $\textit{CBC}(\bfx,\ubfu, \bfxi)$ and $\textit{CLC}(\bfx,\ubfu, \bfxi)$ are affine in $\bfxi$.


We consider a general optimization program:
\begin{align}
\label{eq: general_control_synthesis}
& \min_{\ubfu \in \underline{\calU}}\,\, \|\ubfu - \underline{\bfk}(\bfx)\|^2,  \\
\mathrm{s.t.} \, \,  &\sup_{\mathbb{P}\in \calM_{N}^{r}}\inf_{t \in \mathbb{R}}[\epsilon^{-1}\mathbb{E}_{\bbP}[(G_l (\bfx, \ubfu, \bfxi)+t_l)_{+}] -t_l]  \leq  0, \; \forall l \in [M], \notag
\end{align}
%
%
where $G_l: \calX \times \underline{\calU} \times \Xi \mapsto \bbR$ may represent a safety or stability constraint that is affine in $\bfxi$:
\begin{equation}
\label{eq: G_def}  
G_l (\bfx, \ubfu, \bfxi) =\ubfu^\top \bfq_l(\bfx) + \ubfu^\top \bfR_l(\bfx) \bfxi + \alpha_l(J_l(\bfx)),
\end{equation}
and $J_l$ is used to represent the certificate function (e.g. CLF, CBF). We write $G_l(\bfx, \ubfu, \bfxi_i)$ as $G_l(i)$ for brevity. Depending on the information available about the uncertainty space $\Xi$, we propose two reformulations of \eqref{eq: general_control_synthesis}. In either case, we assume the metric $d$ of $\Xi$ is the Euclidean distance.

\subsection{Reformulation with Unbounded Uncertainty Space}
\label{sec: convex_formulate}
First, we consider the case with no prior knowledge of $\Xi$, meaning that $\Xi = \bbR^k$. We show that the constraints in \eqref{eq: general_control_synthesis} can be reformulated as second-order cone constraints.

\begin{proposition}[\textbf{DRCCP formulation with unbounded support set}]
\label{proposition: convex_final_qp}
Consider the optimization problem in \eqref{eq: general_control_synthesis} with $G_l$ in \eqref{eq: G_def}, $p$-Wasserstein distance with $p=1$, and $\Xi = \bbR^{k}$. Then, the following SOCP is equivalent to \eqref{eq: general_control_synthesis}:
\begin{align}
\label{eq: CBF_Convex_final}
& \min_{\ubfu \in \underline{\calU}, y \in \bbR, t_l \in \bbR, s_l(i) \in \bbR}\,\, y  \\
\mathrm{s.t.} \, \,  
& r \|\ubfu^\top \bfR_l(\bfx)\| + \frac{1}{N} \sum_{i=1}^{N} s_l(i) - t_l\epsilon \leq 0, \notag  \\
& s_l(i) \geq G_l(i)+t_l, \; \; s_l(i) \geq 0, \; \; \forall i \in [N], \;\; \forall l \in [M],  \notag  \\
&y+1 \geq
    \sqrt{ \|2(\ubfu-\underline{\bfk}(\bfx))\|^2  + (y-1)^2}. \notag
\end{align}
\end{proposition}

\begin{proof}
We start by considering the following program:
\begin{align}
\label{eq: CBF_Convex_Reformulation}
& \min_{\ubfu \in \underline{\calU}}\,\, \|\ubfu - \underline{\bfk}(\bfx)\|^2\\
\mathrm{s.t.} \, \,  
& r \|\ubfu^\top \bfR_l(\bfx)\| + \inf_{t \in \mathbb{R}}\left[\frac{1}{N} \sum_{i=1}^{N} (G_l(i) + t_l)_+ - t_l \epsilon\right] \leq 0. \notag  
\end{align}
Based on \cite[Lemma V.8]{Hota2019DataDrivenCC} and assuming $\Xi = \bbR^k$, the supremum over
the Wasserstein ambiguity set (i.e. the constraint in \eqref{eq: general_control_synthesis}) can be written equivalently as the sample average $\inf_{t \in \mathbb{R}}\left[\frac{1}{N} \sum_{i=1}^{N} (G_l(i) + t_l)_+ - t_l \epsilon\right]$ and a regularization term $r L(\ubfu,\bfx)$, where $L(\ubfu,\bfx)$ denotes the Lipschitz constant of $G_l$ in $\bfxi$. 

As defined in \eqref{eq: G_def}, for each $\bfx$, we can define the convex function $L: \underline{\calU} \times \calX \mapsto \mathbb{R}_{>0}$ by
\begin{equation}
\label{eq: L_h_define}
    L(\ubfu, \bfx) = \|\ubfu^\top \bfR_l(\bfx)\|.
\end{equation}
%
Then, the function $\bfxi \mapsto G_l (\bfx, \ubfu, \bfxi)$ is Lipschitz in $\bfxi$ with constant $L(\ubfu, \bfx)$ for fixing $\bfx$ (assuming $L(\ubfu,\bfx) < \infty$).
This is because the Lipschitz constant of a differentiable affine function equals the dual-norm of its gradient \cite{shai_convex}, and the dual norm of the $L_2$ norm is itself. This implies that \eqref{eq: CBF_Convex_Reformulation} is equivalent to \eqref{eq: general_control_synthesis}.

Next, we show that the bi-level optimization in \eqref{eq: CBF_Convex_Reformulation} is equivalent to:
\begin{align}
\label{eq: CBF_Convex_final_qp}
& \min_{\ubfu \in \underline{\calU}, t_l \in \bbR, s_l(i) \in \bbR}\,\, \|\ubfu - \underline{\bfk}(\bfx)\|^2  \\
\mathrm{s.t.} \, \,  
& r \|\ubfu^\top \bfR_l(\bfx)\| + \frac{1}{N} \sum_{i=1}^{N} s_l(i) - t_l \epsilon \leq 0, \notag  \\
& s_l(i) \geq G_l(i)+t_l, \; \; s_l(i) \geq 0, \; \; \forall i \in [N], \; \; \forall l \in [M].  \notag  
\end{align}
For $i \in [N], l \in [M]$, let $(\ubfu_1, t_l^*, s_l(i)^*)$ denote an optimal solution to \eqref{eq: CBF_Convex_final_qp} and $\ubfu_2$ an optimal solution to \eqref{eq: CBF_Convex_Reformulation}, with $\hat{t}_l$ the optimizer for the $\inf$ terms in the constraint of \eqref{eq: CBF_Convex_Reformulation}.

Given $(\ubfu_1, t_l^*, s_l(i)^*)$, we have $s_l(i)^* \geq (G_l(i)+t_l^*)_+$ and
\begin{align}
\label{eq: if_prove}
    &r \|\ubfu_1^\top \bfR_l(\bfx)\| + \frac{1}{N} \sum_{i=1}^{N} s_l(i)^* - t_l^* \epsilon \leq 0.
\end{align}
Thus, if $s_l(i)^*$ is replaced by $(G_l(i)+t_l^*)_+$ in \eqref{eq: if_prove}, we conclude that the constraint in \eqref{eq: CBF_Convex_Reformulation} is satisfied with $\ubfu_1$ and $t_l^*$.
%
%
This implies that $\ubfu_1$ is also a solution to \eqref{eq: CBF_Convex_Reformulation}, and the cost satisfies $\|\ubfu_1 - \underline{\bfk}(\bfx)\|^2 \geq \|\ubfu_2 - \underline{\bfk}(\bfx)\|^2 $.


Given $\ubfu_2$ and $\hat{t}_l$, for every $i \in [N]$, we choose 
$\hat{s}_l(i) = (G_l(i) + \hat{t}_l)_+$.
%
This implies $\hat{s}_l(i) \geq G_l(i) + \hat{t}_l$,  $\hat{s}_l(i) \geq 0$, and the first constraints in \eqref{eq: CBF_Convex_final_qp} is satisfied since
\begin{align*}
    &r \|\ubfu_2^\top \bfR_l(\bfx)\| + \frac{1}{N} \sum_{i=1}^{N} (G_l(i) + \hat{t}_l )_+ - \hat{t}_l \epsilon \leq 0.
\end{align*}
Thus, $(\ubfu_2,\hat{t}, \hat{s_l(i)})$ is also a solution to \eqref{eq: CBF_Convex_final_qp}. Furthermore, the cost satisfies $\|\ubfu_1 - \underline{\bfk}(\bfx)\|^2 \leq \|\ubfu_2 - \underline{\bfk}(\bfx)\|^2  $ since $(\ubfu_1, t_l^*,s_l(i)^*)$ is an optimal solution to \eqref{eq: CBF_Convex_final_qp}. Therefore, both costs are equal, and \eqref{eq: CBF_Convex_final_qp} and \eqref{eq: CBF_Convex_Reformulation} are equivalent.

Finally, by reformulating the objective function of \eqref{eq: CBF_Convex_final_qp} as a linear objective with an SOC constraint \cite[Proposition IV.3]{Long2022RAL}, we conclude the SOCP \eqref{eq: CBF_Convex_final} is equivalent to \eqref{eq: general_control_synthesis}.
\end{proof}

Proposition~\ref{proposition: convex_final_qp} allows control synthesis with distributionally robust safety and stability constraints without prior knowledge about the uncertainty support set $\Xi$. The SOCP in \eqref{eq: CBF_Convex_final} can be solved efficiently online using an off-the-shelf solver (e.g. \cite{mosek}).

\subsection{Reformulation with Bounded Uncertainty Space}

Assuming no prior knowledge about the uncertainty set $\Xi$ may result in an overly conservative controller. This motivates us to also consider the case that the uncertainty support set $\Xi$ is a compact polyhedron.




\begin{proposition}[\textbf{DRCCP formulation with bounded polyhedron support set}]
\label{proposition: affine_final_socp}
Consider the optimization problem in \eqref{eq: general_control_synthesis} with $G_l$ in \eqref{eq: G_def}, $p$-Wasserstein distance with $p=1$, and compact $\Xi = \{\bfxi \in \mathbb{R}^k \;|\; \bfC \bfxi \leq \bfd\}$, where $\bfC \in \mathbb{R}^{q \times k}$ and $\bfd \in \mathbb{R}^q$ for some $q > 0$. Then, the following SOCP is equivalent to \eqref{eq: general_control_synthesis},
\begin{align}
\label{eq: CBF_Affine_Reformulation_final}
& \min_{\ubfu \in \underline{\calU}, y \in \mathbb{R}, t_l \in \bbR, s_l(i) \in \bbR, \beta \in \bbR_{\geq 0},  \bfeta_{i} \in \bbR^q} \, \, y \qquad \qquad \\
\mathrm{s.t.} \, \, 
& \beta r + \frac{1}{N}\sum_{i=1}^N s_l(i) - t_l\epsilon \leq 0, \notag\\
& s_l(i) \geq 0, \notag \\
& s_l(i) \geq \ubfu^\top \bfq_l(\bfx) + t_l + (\ubfu^\top \bfR_l(\bfx) - \bfeta_{i}^\top \bfC) \bfxi_i + \bfeta_{i}^\top \bfd + \notag \\
&\alpha_l(J_l(\bfx)), \notag \\
& \|\ubfu^\top \bfR_l(\bfx) - \bfeta_{i}^\top \bfC \| \leq \beta, \; \; \bfeta_{i} \geq \boldsymbol{0}_q, \; \; \forall i \in [N], \; \;\forall l \in [M], \notag \\
&y+1 \geq
    \sqrt{ \|2(\ubfu-\underline{\bfk}(\bfx))\|^2  + (y-1)^2}. \notag
\end{align}
\end{proposition}
\begin{proof}
Based on \cite[Proposition V.1]{Hota2019DataDrivenCC} and \cite[Corollary 5.1]{Esfahani2018DatadrivenDR}, we know the following program is equivalent to \eqref{eq: general_control_synthesis},
\begin{align}
\label{eq: CBF_Affine_Reformulation}
& \min_{\ubfu \in \underline{\calU},  t_l \in \bbR, s_l(i) \in \bbR, \beta \in \bbR_{\geq 0},  \bfeta_i \in \bbR^q} \, \,\|\ubfu - \underline{\bfk}(\bfx)\|^2 \\
\mathrm{s.t.} \, \, 
& \beta r + \frac{1}{N}\sum_{i=1}^N s_l(i) - t_l \epsilon \leq 0, \notag\\
& (\ubfu^\top \bfq_l(\bfx) + t_l + (\ubfu^\top \bfR_l(\bfx) - \bfeta_{i}^\top \bfC ) \bfxi_i + \bfeta_{i}^\top \bfd +\notag\\ 
& \alpha_l(J_l(\bfx)))_+ \leq s_l(i), \notag \\
& \|\ubfu^\top \bfR_l(\bfx) - \bfeta_{i}^\top \bfC \| \leq \beta, \quad \bfeta_{i} \geq \boldsymbol{0}_q, \quad \forall i \in [N] \notag 
\end{align}
Next, we aim to rewrite \eqref{eq: CBF_Affine_Reformulation} as a SOCP. The ReLU-type inequality 
$(w_i)_+ = (\ubfu^\top \bfq_l(\bfx) + t + (\ubfu^\top \bfR_l(\bfx) - \bfeta_i^\top \bfC) \bfxi_i + \bfeta_i^\top \bfd + \alpha_l(J_l(\bfx)))_+ \leq s_l(i)$
%
%
can be written equivalently as two constraints: 
%
    $s_l(i) \geq w_i$ and $s_l(i) \geq 0$.
%
Following the same technique as in Proposition~\ref{proposition: convex_final_qp}, we conclude that \eqref{eq: CBF_Affine_Reformulation_final} is equivalent to \eqref{eq: general_control_synthesis}.
\end{proof}


\begin{remark}[\textbf{Comparison between the two formulations}] 
{\rm If $\bfC = \mathbf{0}_{q \times k}$ and $\bfd = \mathbf{0}_q$ in Proposition~\ref{proposition: affine_final_socp}, then $\Xi = \bbR^k$ and the SOCP in \eqref{eq: CBF_Affine_Reformulation_final} reduces to \eqref{eq: CBF_Convex_final}. \hfill $\bullet$}
\end{remark}

\begin{remark}[\textbf{Different choice of metric $d$}]
{\rm If instead of $L_2$ norm, we take the metric $d$ of $\Xi$ to be the $L_1$ norm, then the optimization problems in Propositions~\ref{proposition: convex_final_qp} and \ref{proposition: affine_final_socp} become QPs. Details are provided in Appendix~\ref{sec: appendix_l1_qp}. \hfill $\bullet$}
\end{remark}

%% file: tex/Experiment_formulation.tex
\section{Evaluation}
\label{sec: experiment}

We evaluate the proposed distributionally robust approach for safe and stable control synthesis in an adaptive cruise control problem introduced in \cite{ames2014cdc}.

\subsection{Cruise Control Model}\label{sec: evaluate_model}
Consider a simplified adaptive cruise control model that consists of two vehicles, one leading vehicle traveling at a constant speed and one following vehicle using our control synthesis methodology. The objective is have the following vehicle achieve a desired speed while keeping a safe distance from the leading vehicle. The system model is:
\begin{equation}\label{ACC: dynamic}
\begin{aligned}
    \underbrace{\begin{bmatrix}
    \dot{p} \\
    \dot{v} \\
    \dot{z}
    \end{bmatrix}}_{\dot{\bfx}}
	=	
	\underbrace{\left[ \begin{matrix}
			v \\
			- \frac{1}{m}F_r(v) \\
			v_0-v
			\end{matrix} \right]}_{f(\bfx)}
	+
	\underbrace{\left[ \begin{matrix}
		0 \\
		\frac{1}{m}\\
		0
		\end{matrix} \right]}_{g(\bfx)}u,
\end{aligned}
\end{equation}
where $v$ and $v_0$ are the velocities of the following and leading vehicles, respectively, $F_r(v) = f_0+f_1v+f_2v^2$ is the air drag, $p$ is the following vehicle position, and $z$ is the distance to the leading vehicle. The input is constrained by 
%
$-c_dg \leq \frac{u}{m} \leq c_ag$, 
%
where $c_d$ and $c_a$ denote the factor of $g$ for deceleration and acceleration, respectively. We define a CLF,
%
$V(\bfx) = (v-v_d)^2$, 
%
where $v_d$ is the desired speed of the following vehicle. The safety requirements is specified by the CBF
%
$h(\bfx) = z - \frac{1}{2} \frac{(v_0 - v)^2}{c_d g} - 1.8v$.
%
%
%
We assume that the system \eqref{ACC: dynamic} is uncertain with the following parametric uncertainty, 
\begin{equation}\label{ACC: dynamics with uncertainty}
	\dot{\bfx} = (\tilde{F}(\bfx) + \sum_{i=1}^3 W_i(\bfx) \xi_i)\ubfu
\end{equation}
where $\tilde{F}(\bfx) = [f(\bfx) \ g(\bfx)], \ \ubfu = [1 \ u]^\top$, and:
\begin{equation} \notag
\begin{aligned}
	W_1(\bfx)
	\!=\!	
	\begin{bmatrix}
		0 & 0\\
		\frac{v}{20} & 0 \\
		0 & 0
		\end{bmatrix}\!\!,\, W_2(\bfx)
		\!=\!	
		\begin{bmatrix}
		0 & 0\\
		0 & \frac{0.05}{m} \\
		0 & 0
		\end{bmatrix}\!\!,\,  W_3(\bfx)
		\!=\!	
		\begin{bmatrix}
		0 & 0\\
		0 & 0 \\
		\frac{2z}{25} & 0
		\end{bmatrix}\!\!,
\end{aligned}
\end{equation}
where $W_1$, $W_2$, and $W_3$ represent the model perturbations in the drag, input force, and leading vehicle distance, respectively. Table~\ref{table: parameters} reports the parameter values used in the simulation. 

\begin{table} 
\centering
\caption{Parameters used in the simulation results}
\begin{tabular}{ccc}
	\toprule 
	Variable & Description & Value\\
	\midrule  
	$g$& Gravitational acceleration & 9.81\\
	$m$& Mass of vehicle & 1650\\
	$f_0$& Coefficient in $F_r(v)$ & 0.1\\
	$f_1$& Coefficient in $F_r(v)$ & 5\\
	$f_2$& Coefficient in $F_r(v)$ & 0.25\\
	$v_d$& Desired speed & 35\\
	$v_0$& Speed of leading vehicle & 20\\
	$c_a$& Max accelerate constant & 0.3\\
	$c_d$& Max decelerate constant & -0.3\\
	\bottomrule
\end{tabular}
\label{table: parameters}
\end{table}

\subsection{Results}\label{sec: experiment_results}

We evaluate our distributionally robust control synthesis approach and illustrate its versatility in handling model uncertainty. We report simulation results from the unbounded uncertainty formulation (Proposition~\ref{proposition: convex_final_qp}) and the bounded uncertainty formulation (Proposition~\ref{proposition: affine_final_socp}). For comparison, we include results from the CLF-CBF-QP (which takes no model uncertainty into account) formulation in \cite{ames2014cdc} with baseline controller $\underline{\bfk}(\bfx) = [1 \ F_r(v)]^\top$, the robust (which requires prior knowledge on the error bound) and the chance-constrained (which assumes the uncertainty distribution to be Gaussian) formulations in \cite{Long2022RAL}. 
In the simulation, the error bounds are provided by the support set information and the Gaussian parameters are estimated via offline uncertainty samples.
%
%
In all cases, we use the value of the CBF as a measure of the safety ensured by the corresponding approach. We consider different choices of Wasserstein radius $r_1 = r_2 = r$, confidence level $\epsilon$, support set $\Xi$, offline uncertainty samples $\{\bfxi_i\}_{i=1}^N$, and online true uncertainty realization $\bfxi^{*}$. To demonstrate that our formulation ensures safety for out-of-sample uncertainty, we use different distributions for sampling offline observations $\bfxi_i$ and a true online uncertainty realization $\bfxi^*$. 

We consider $8$ cases with different parameter choices, where $\calN$ and $\calB$ denote normal and beta distributions, respectively. For each case, we conduct 50 simulations with the same $\{\bfxi_i\}_{i=1}^{N = 10}$ and different $\bfxi^{*}$.
\\
\textit{Case 1 (Gaussian Distribution):} $r = 0.3, \ \epsilon = 0.1, \ \Xi = [\mathbf{-2}_3,\mathbf{2}_3], \ \bfxi_i \sim \calN(-1,1/3), \ \bfxi^{*} \sim \calN(-1,1/3)$. \\ 
\textit{Case 2 (Confident in Sample):} $r = 0.001, \ \epsilon = 0.1, \ \Xi = [\mathbf{-2}_3,\mathbf{2}_3], \ \bfxi_i \sim \calN(-1,1/3), \ \bfxi^{*} \sim \calN(-1,1/3)$. \\ 
\textit{Case 3 (Out of Sample):} $r = 0.3, \ \epsilon = 0.1, \ \Xi = [\mathbf{-2}_3,\mathbf{2}_3], \ \bfxi_i \sim \calN(0,0.2), \ \bfxi^{*} \sim \calB(0.1,2) - 1$. \\ 
\textit{Case 4 (Baseline Radius and Confidence):} $r = 0.3, \ \epsilon = 0.1, \ \Xi = [\mathbf{-2}_3,\mathbf{2}_3], \ \bfxi_i \sim 4\calB(3,0.1) - 2, \ \bfxi^{*} \sim \calN(-1,1/3)$. \\ 
\textit{Case 5 (Larger Radius):} $r =  0.5, \ \epsilon = 0.1, \ \Xi = [\mathbf{-2}_3,\mathbf{2}_3], \ \bfxi_i \sim 4\calB(3,0.1) - 2, \ \bfxi^{*} \sim \calN(-1,1/3)$. \\ 
\textit{Case 6 (Higher Confidence):} $r =0.3, \ \epsilon = 0.05, \ \Xi = [\mathbf{-2}_3,\mathbf{2}_3], \ \bfxi_i \sim 4\calB(3,0.1) - 2, \ \bfxi^{*} \sim \calN(-1,1/3)$. \\ 
\textit{Case 7 (Larger Radius and Higher Confidence):} $r = 0.5, \ \epsilon= 0.05, \ \Xi = [\mathbf{-2}_3,\mathbf{2}_3], \ \bfxi_i \sim 4\calB(3,0.1) - 2, \ \bfxi^{*} \sim \calN(-1,1/3)$. \\ 
\textit{Case 8 (Out of Support):} $r =0.3, \ \epsilon = 0.1, \ \Xi = [\mathbf{-0.5}_3,\mathbf{0.5}_3], \ \bfxi_i \sim \calB(2,0.1) - 0.5, \ \bfxi^{*} \sim \calN(-1,1/3)$. 

In Table~\ref{table: results}, we report the failure rate and the average CBF values for the $8$ cases above. In \textit{Cases 1} and \textit{2}, under Gaussian uncertainty in the dynamics model, all formulations ensure safety except the CLF-CBF-QP. When we set the Wasserstein radius small ($r = 0.001$), meaning that we are confident in the offline uncertainty samples, the unbounded DRCCP and bounded DRCCP formulations have the same mean CBF values. 
%
%
In \textit{Case 3}, we verify that if the uncertainty distribution shifts during the online phase (e.g., the online uncertainty no longer from a Gaussian distribution), then the Gaussian CLF-CBF-SOCP formulation fails, while the other three formulations ensure safety. \textit{Cases 4} to \textit{7} demonstrate the effects of the Wasserstein distance and confidence level in our bounded and unbounded DRCCP formulations. On the one hand, the unbounded DRCCP formulations tend to be more conservative if we increase the Wasserstein radius $r$ and/or the confidence level, as shown in Fig.~\ref{fig: out_of_sample}. On the other hand, only increasing the confidence level makes the bounded DRCCP controller more conservative, since support information provides a tighter bound than the Wasserstein radius. In \textit{Case 8}, we see that the unbounded DRCCP formulation works well even with out-of-support uncertainty, while the robust CLF-CBF-SOCP and bounded DRCCP both fail due to the provided incorrect support set information, as Fig.~\ref{fig: out_of_support} shows. 

Generally, the controller provided by the bounded DRCCP formulation has the best performance in ensuring safety while not being too conservative (smaller average CBF values). However, if one fails to provide reliable support set information, then the controller provided by the unbounded DRCCP formulation is the safe choice.

\begin{figure}[t]
    \centering
    \includegraphics[width=\linewidth]{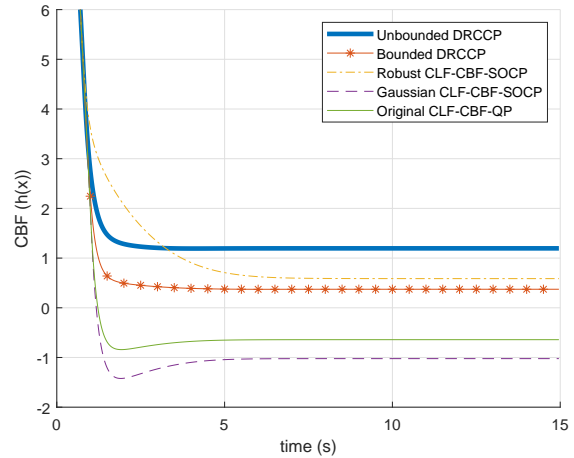}
    \caption{CBF value of one of the 50 simulations corresponding to \textit{Case 5}. Both the unbounded and bounded DRCCP formulations ensure safety while the CLF-CBF-QP and the Gaussian formulation fail. This demonstrates that either the bounded or unbounded DRCCP formulation ensures safety for out-of-sample uncertainty. The unbounded DRCCP  formulation is more conservative since it does not take the uncertainty support set information into account.}
    \label{fig: out_of_sample}
\end{figure}

%
%

\begin{figure}[t]
    \centering
    \includegraphics[width=\linewidth]{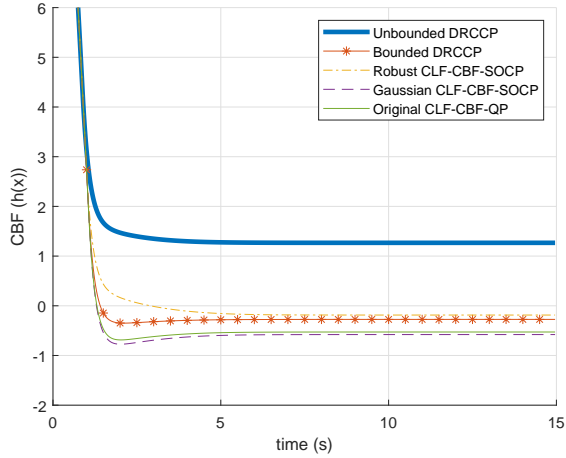}
    \caption{CBF value of one of the 50 simulations corresponding to \textit{Case 8}. The offline uncertainty distribution is set to be within the uncertainty support set $\Xi$, while the online uncertainty distribution is outside of $\Xi$. The controller obtained with the bounded DRCCP formulation~\eqref{eq: CBF_Affine_Reformulation_final} fails to guarantee safety because the assumptions in Proposition~\ref{proposition: affine_final_socp} are violated. However, the controller obtained with the unbounded DRCCP formulation~\eqref{eq: CBF_Convex_final} still guarantees safety.
    }
    \label{fig: out_of_support}
\end{figure}
	
%
%



%
\begin{table}[htb]
	\centering
	\caption{Failure rate and average CBF values. The results are shown in the following format: $a\% \mid b$, where $a\%$ denotes the violation rate of each formulation: (simulations with unsafe state)/(total simulations), and $b$ denotes the average value of CBF over all simulations. The average CBF value is computed based on stabilized CBF values, e.g., for $5 \leq t \leq 15$ in Fig.~\ref{fig: out_of_support}.
    }
	\scalebox{0.8}{
    \begin{tabular}{cccccc}
		\toprule 
        \multirow{ 2}{*}{Case} & Unbounded & Bounded & Robust & Gaussian & Original \\
         & DRCCP & DRCCP & CLF-CBF-SOCP & CLF-CBF-SOCP & CLF-CBF-QP \\
		\midrule  
		1 & 0\% $\mid$ 2.11 & $\boldsymbol{0\%}$ $\mid$ $\bf0.56$ & 0\% $\mid$ 1.01 & 0\% $\mid$ 0.55 & 100\% $\mid$ -0.48 \\
        2 & $\boldsymbol{0\%}$ $\mid$ $\bf0.35$ & $\boldsymbol{0\%}$ $\mid$ $\bf0.35$ & 0\% $\mid$ 1.01 & 0\% $\mid$ 0.56 & 100\% $\mid$ -0.48 \\
		3 & 0\% $\mid$ 1.23 & $\boldsymbol{0\%}$ $\mid$ $\bf0.55$ & 0\% $\mid$ 0.96 & 98\% $\mid$ -0.17 & 100\% $\mid$ -0.50 \\
		4 & 0\% $\mid$ 0.56 & $\boldsymbol{0\%}$ $\mid$ $\bf0.51$ & 0\% $\mid$ 0.89 & 100\% $\mid$ -0.92 & 100\% $\mid$ -0.53 \\
		5 & 0\% $\mid$ 1.42 & $\boldsymbol{0\%}$ $\mid$ $\bf0.57$ & 0\% $\mid$ 1.01 & 100\% $\mid$ -0.88 & 100\% $\mid$ -0.48 \\
		6 & 0\% $\mid$ 1.94 & $\boldsymbol{0\%}$ $\mid$ $\bf0.57$ & 0\% $\mid$ 1.01 & 100\% $\mid$ -0.65 & 100\% $\mid$ -0.48 \\
		7 & 0\% $\mid$ 4.49 & $\boldsymbol{0\%}$ $\mid$ $\bf0.57$ & 0\% $\mid$ 1.01 & 100\% $\mid$ -0.64 & 100\% $\mid$ -0.49 \\
		8 & $\boldsymbol{0\%}$ $\mid$ $\bf1.29$ & 98\% $\mid$ -0.26 & 84\% $\mid$ -0.16 & 100\% $\mid$ -0.57 & 100\% $\mid$ -0.51 \\
		\bottomrule
	\end{tabular}}
	\label{table: results}
\end{table}
%

%% file: tex/Conclusion.tex
\section{Conclusions}

We considered the problem of enforcing safety and stability of uncertain control-affine systems. Compared with previous approaches, we derive new distributionally robust chance constrained formulations of safe and stable control synthesis that do not require any prior knowledge of error bounds or uncertainty distributions.  Using only offline model uncertainty samples, we show that our formulations ensure safety and stability with out-of-sample errors during online execution. Future work will consider deploying the algorithms on real autonomous systems and learning the perturbation matrices and uncertainty samples from offline state-control sequences. 

%% file: tex/Appendix_qp.tex
\section{Different Choice of Metric $d$}\label{sec: appendix_l1_qp}
We show that when the metric $d$ of the uncertainty support set $\Xi$ is the $L_1$ norm (instead of the Euclidean norm as in Propositions~\ref{proposition: convex_final_qp} and~\ref{proposition: affine_final_socp}), then \eqref{eq: general_control_synthesis} becomes a QP for both the cases of unbounded and bounded uncertainty sets.


\begin{proposition}[\textbf{DRCCP formulation with unbounded support set under $L_1$ norm}]
\label{proposition: convex_l1_qp}
Consider the optimization problem in \eqref{eq: general_control_synthesis} with $G_l$ in \eqref{eq: G_def}, $p$-Wasserstein distance with $p=1$, and $\Xi = \bbR^{k}$ with metric $d(\bfxi,\bfxi') = \|\bfxi-\bfxi'\|_1$. Then, the following QP is equivalent to \eqref{eq: general_control_synthesis}:
\begin{align}
\label{eq: CBF_Convex_l1}
& \min_{\ubfu \in \underline{\calU}, t_l \in \bbR, s_l(i) \in \bbR}\,\, \|(\ubfu-\underline{\bfk}(\bfx))\|^2  \\
\mathrm{s.t.} \, \,  
& r |\bfR_l^\top(\bfx) \ubfu| \leq (t_l \epsilon - \frac{1}{N} \sum_{i=1}^{N} s_l(i)) \mathbf{1}_{k}, \notag  \\
& s_l(i) \geq G_l(i)+t_l, \; \; s_l(i) \geq 0, \; \; \forall i \in [N], l \in [M].  \notag
\end{align}
\end{proposition}
%


\begin{proposition}[\textbf{DRCCP formulation with bounded polyhedron set under $L_1$ norm}]
\label{proposition: affine_l1_qp}
Consider the optimization problem in \eqref{eq: general_control_synthesis} with $G_l$ in \eqref{eq: G_def}, $p$-Wasserstein distance with $p=1$, and compact $\Xi = \{\bfxi \in \mathbb{R}^k \;|\; \bfC \bfxi \leq \bfd\}$ with $\bfC \in \mathbb{R}^{q \times k}$ and $\bfd \in \mathbb{R}^q$ for some $q>0$ and metric $d(\bfxi,\bfxi') = \|\bfxi-\bfxi'\|_1$. Then, the following QP is equivalent to \eqref{eq: general_control_synthesis}:  
\begin{align}
\label{eq: CBF_Affine_Reformulation_l1}
& \min_{\ubfu \in \underline{\calU}, t_l \in \bbR, s_l(i) \in \bbR, \beta \in \bbR_{\geq 0},  \bfeta_{i} \in \bbR^q} \, \, \|(\ubfu-\underline{\bfk}(\bfx))\|^2 \qquad  \\
\mathrm{s.t.} \, \,  
& \beta r + \frac{1}{N}\sum_{i=1}^N s_l(i) - t_l\epsilon \leq 0, \notag\\
& s_l(i) \geq 0, \notag \\
& s_l(i) \geq \ubfu^\top \bfq_l(\bfx) + t_l + (\ubfu^\top \bfR_l(\bfx) - \bfeta_{i}^\top \bfC) \bfxi_i + \bfeta_{i}^\top \bfd + \notag \\
&\alpha_l(J_l(\bfx)), \notag \\
& |\bfR_l^\top(\bfx)\ubfu - \bfC^\top \bfeta_{i} | \leq \beta \mathbf{1}_{k}, \; \; \bfeta_{i} \geq \boldsymbol{0}_q, \; \; \forall i \in [N], l \in [M]. \notag
\end{align}
\end{proposition}


We provide a proof sketch for these results. When $d$ is the $L_1$ norm, the Lipschitz constant in \eqref{eq: L_h_define} is defined by the $L_{\infty}$ norm, since the dual norm of $L_1$ is $L_{\infty}$. Similarly, the $L_2$ norm in the fourth constraint in Proposition~\ref{proposition: affine_final_socp} is replaced by the $L_{\infty}$ norm. This means that we no longer need to reformulate the objective function, since all constraints are linear in the decision variables and both optimization problems are QPs.